\documentclass[11pt]{article}
\usepackage[a4paper,margin=1in]{geometry}
\usepackage{amsmath,amssymb,amsthm}
\usepackage{enumitem}
\usepackage{booktabs}
\usepackage{hyperref}
\usepackage{microtype}

\hypersetup{colorlinks=true,linkcolor=blue,citecolor=blue,urlcolor=blue}

\newtheorem{theorem}{Theorem}[section]
\newtheorem{proposition}[theorem]{Proposition}
\newtheorem{lemma}[theorem]{Lemma}

\newtheorem{definition}[theorem]{Definition}
\newtheorem{remark}[theorem]{Remark}

\title{Euler Constraints and the Cubic Criticality of Complete Bipartite Preference Structures}
\author{Yoshiteru Ishida}
\date{}

\begin{document}
\maketitle

\begin{abstract}
We introduce a polyhedral realizability problem for complete bipartite preference structures, the two-sided strict preference profiles that underlie the stable marriage problem. Given two parts $M$ and $W$ of size $k$, a strong polyhedral realization asks for a bipartite polyhedral graph with bipartition classes of size $k$ in which every vertex is adjacent to all but one vertex on the opposite side. This condition attempts to represent complete opposite-side preference capacity by direct polyhedral adjacency, with one geometrically exceptional opposite-side vertex for each agent. We prove an Euler-type criticality theorem: such a strong polyhedral realization exists if and only if $k=4$, and in that case the underlying graph is the cube graph $Q_3$, equivalently $K_{4,4}$ minus a perfect matching. The proof is the collision of the strong edge requirement $E=k(k-1)$ with the bipartite planar bound $E\le 4k-4$ and the minimum-degree constraint for polyhedral graphs. We then distinguish weak realizations. A family of $(k-1)$-gonal trapezohedra gives balanced bipartite polyhedral graphs with $2k$ vertices and $4k-4$ edges for all $k\ge4$, showing that weak maximal realizability persists beyond the cubic critical case. However, the Euler interval $3k\le E\le4k-4$ is not fully realizable: we give a direct proof that no bipartite polyhedral graph exists with $V=10$ and $E=15$, equivalently that there is no even triangulation of the sphere on seven vertices. Finally, in the cube case, we show that cube-distance compatibility of a size-four preference profile is equivalent to the existence of a perfect matching of mutually last-ranked pairs. This gives a first preference-theoretic manifestation of the cubic criticality.
\end{abstract}


\section{Introduction}

Euler's formula imposes severe capacity constraints on planar and polyhedral graphs. In this paper we study a small but structurally suggestive instance of this phenomenon: the attempt to realize complete bipartite preference structures by bipartite polyhedral graphs.

Let $M$ and $W$ be two disjoint sets, each of size $k$. A complete bipartite preference structure consists of a strict total ordering of $W$ for each element of $M$, and a strict total ordering of $M$ for each element of $W$. Such structures are the input profiles of the classical stable marriage problem, but the present paper focuses first on their polyhedral realizability rather than on stability algorithms.

The guiding question is the following. Can the complete two-sided nature of such a profile be represented by a three-dimensional bipartite polyhedral graph? The strongest version of this question asks for every vertex to be adjacent to all but one vertex of the opposite part. We call this a strong polyhedral realization. It creates a direct geometric analogue of the complete opposite-side set: for each vertex, all but one opposite-side agents are geometrically adjacent, while one is geometrically exceptional.

The cube immediately suggests itself. The cube graph $Q_3$ has eight vertices, is bipartite with four vertices in each part, and every vertex is adjacent to three vertices on the opposite side. Thus, for $k=4$, the cube realizes the pattern
\[
\text{four opposite-side vertices} = \text{three adjacent vertices} + \text{one non-adjacent vertex}.
\]
Our main theorem says that this is not an accident. The cube is the unique strong polyhedral realization.

The key numerical comparison is
\[
E_{\mathrm{strong}}=k(k-1),\qquad E_{\mathrm{bipartite\ planar}}\le 4k-4.
\]
The strong condition demands $k(k-1)$ edges, while bipartite planarity on $2k$ vertices permits at most $4k-4$ edges. Combined with the polyhedral minimum-degree condition, this forces $k=4$. We use the phrase \emph{cubic criticality} because the critical graph is simultaneously the cube and a 3-regular, or cubic, graph.

The paper also separates strong and weak realizability. Strong realizability collapses to the cube. Weak realizability, by contrast, persists: the $(k-1)$-gonal trapezohedron gives a balanced bipartite polyhedral graph with $2k$ vertices and $4k-4$ edges for every $k\ge4$. This family includes the cube at $k=4$.

At the same time, the Euler interval
\[
3k\le E\le 4k-4
\]
for balanced bipartite polyhedral graphs is not a full existence interval. We prove that the first nontrivial intermediate case, $k=5$ and $E=15$, is impossible. Under planar duality this is the small case asserting that there is no simple even triangulation of the sphere on seven vertices; this observation places the example within the standard theory of even triangulations and cubic bipartite planar graphs. Thus, beyond the main cubic criticality, polyhedral realizability has finer obstructions not captured by Euler inequalities alone.

Finally, we return to preference rankings. In the cube case, the unique opposite-side non-neighbor of each vertex may be interpreted as a geometrically farthest opposite-side agent. A profile is cube-distance-compatible if these non-neighbor pairs are exactly the mutually last-ranked pairs. We prove that this is equivalent to the existence of a perfect matching of mutually last-ranked pairs.

\section{Complete bipartite preference structures}

\begin{definition}[Complete bipartite preference structure]
Let $M$ and $W$ be disjoint finite sets with $|M|=|W|=k$. A complete bipartite preference structure of size $k$ is a tuple
\[
I=(M,W,\{\succ_m\}_{m\in M},\{\succ_w\}_{w\in W}),
\]
where each $\succ_m$ is a strict total order on $W$, and each $\succ_w$ is a strict total order on $M$.
\end{definition}

Equivalently, one may regard $I$ as a complete bipartite graph $K_{k,k}$ whose edge $(m,w)$ carries two directed rank labels,
\[
(r_m(w),r_w(m)),
\]
where $r_m(w)$ is the rank of $w$ in $m$'s list, and $r_w(m)$ is the rank of $m$ in $w$'s list.

The main graph-theoretic realization problem in this paper concerns the underlying two-part agent sets $M$ and $W$. The preference rankings are brought back in Section~\ref{sec:preference-manifestation}, where non-adjacency in the cube is related to last-ranked pairs.

\section{Bipartite polyhedral graphs and the Euler interval}

A polyhedral graph is the graph of a convex three-dimensional polyhedron. Equivalently, by Steinitz's theorem, it is a planar 3-connected graph. We consider bipartite polyhedral graphs with equal bipartition classes.

\begin{lemma}[Bipartite planar edge bound]
Let $G$ be a bipartite planar graph with $V\ge 3$ vertices and $E$ edges. Then
\[
E\le 2V-4.
\]
\end{lemma}

\begin{proof}
Every face in a bipartite plane graph has length at least four. Hence $4F\le 2E$. Euler's formula $V-E+F=2$ gives $E\le 2V-4$.
\end{proof}

\begin{lemma}[Polyhedral minimum degree]
Every polyhedral graph has minimum degree at least three.
\end{lemma}

\begin{proof}
A polyhedral graph is 3-connected, hence no vertex can have degree at most two.
\end{proof}

\begin{proposition}[Euler interval]
Let $G$ be a bipartite polyhedral graph with bipartition classes of size $k$ and edge number $E$. Then
\[
\boxed{3k\le E\le 4k-4.}
\]
\end{proposition}

\begin{proof}
The graph has $2k$ vertices and minimum degree at least three, so $2E\ge 3(2k)$ and $E\ge3k$. The bipartite planar edge bound gives $E\le2(2k)-4=4k-4$.
\end{proof}

The interval is nonempty precisely when $k\ge4$. Thus $k=4$ is the first possible size for balanced bipartite polyhedral graphs.

\section{Strong polyhedral realization}

\begin{definition}[Strong polyhedral realization]
Let $I$ be a complete bipartite preference structure of size $k$. A strong polyhedral realization of its underlying two-sided agent structure consists of a bipartite polyhedral graph
\[
G=(V_M\sqcup V_W,E)
\]
with $|V_M|=|V_W|=k$, together with bijections $M\to V_M$ and $W\to V_W$, such that each vertex is adjacent to exactly $k-1$ vertices of the opposite bipartition class.
\end{definition}

Thus each agent has exactly one geometrically non-adjacent opposite-side agent. The preference rankings are not required to determine the realization at this stage; the strong realization condition concerns the polyhedral capacity of the underlying complete bipartite agent structure.

Under the strong condition the graph is $(k-1)$-regular on $2k$ vertices, hence
\[
E_{\mathrm{strong}}=k(k-1).
\]
Since non-adjacency is symmetric and each vertex misses exactly one opposite-side vertex, the missing opposite-side pairs form a perfect matching. Equivalently, the underlying graph is $K_{k,k}$ with a perfect matching deleted.

\begin{theorem}[Cubic criticality theorem]
A strong polyhedral realization exists if and only if $k=4$. In that case, the underlying graph is isomorphic to the cube graph $Q_3$.
\end{theorem}

\begin{proof}
In a strong realization, $E=k(k-1)$. Since the graph is bipartite planar on $2k$ vertices,
\[
k(k-1)=E\le4k-4,
\]
so $(k-1)(k-4)\le0$ and therefore $k\le4$ for positive integer $k$. Since the graph is polyhedral, its minimum degree is at least three. But the strong degree is $k-1$, so $k-1\ge3$, hence $k\ge4$. Thus $k=4$.

For $k=4$, the graph is $K_{4,4}$ minus a perfect matching. This graph is isomorphic to $Q_3$ by Proposition~\ref{prop:k44cube}. Conversely, $Q_3$ is bipartite with four vertices in each part and degree three, and hence gives a strong realization for $k=4$.
\end{proof}

\begin{proposition}\label{prop:k44cube}
For any perfect matching $P$ of $K_{4,4}$,
\[
K_{4,4}\setminus P\cong Q_3.
\]
Equivalently, the crown graph $S^0_4$ is the cube graph.
\end{proposition}

\begin{proof}
Label the vertices of $Q_3$ by binary strings in $\{0,1\}^3$. Let $A$ be the set of even-parity strings and $B$ the set of odd-parity strings. The cube graph is bipartite with parts $A$ and $B$, and an even string is adjacent to the three odd strings obtained by changing exactly one coordinate. The unique odd string not adjacent to a given even string is its bitwise complement, which differs in all three coordinates. Thus the missing opposite-side edges form a perfect matching. Hence $Q_3$ is $K_{4,4}$ minus a perfect matching. Since all perfect matchings in $K_{4,4}$ are equivalent under relabeling, the conclusion follows.
\end{proof}

\section{Weak maximal realizations by trapezohedra}

Strong realization is unique to $k=4$, but weaker balanced bipartite polyhedral realizations exist for every $k\ge4$.

\begin{definition}[Weak polyhedral realization]
A weak polyhedral realization of a complete bipartite preference structure of size $k$ consists of a bipartite polyhedral graph with two bipartition classes of size $k$, together with bijections from $M$ and $W$ to the two classes. No requirement is imposed that each vertex be adjacent to $k-1$ opposite-side vertices.
\end{definition}

The following family attains the upper edge bound in the Euler interval.

\begin{definition}[$n$-gonal trapezohedron graph]
For $n\ge3$, let $T_n$ be the graph with vertices
\[
\{N,S,x_0,x_1,\ldots,x_{2n-1}\},
\]
where $x_0x_1\cdots x_{2n-1}x_0$ is a cycle, $N$ is adjacent to $x_i$ for even $i$, and $S$ is adjacent to $x_i$ for odd $i$. Indices are taken modulo $2n$.
\end{definition}

The graph $T_n$ is the graph of the standard $n$-gonal trapezohedron, the dual of the $n$-gonal antiprism; see, for example, the standard polytope references on convex polyhedra and their duals \cite{Grunbaum2003,Ziegler1995,Cromwell1997}. It is a quadrangulation: its faces are the cycles
\[
N-x_{2j}-x_{2j+1}-x_{2j+2}-N
\]
and
\[
S-x_{2j+1}-x_{2j+2}-x_{2j+3}-S,
\]
for $j=0,\ldots,n-1$.

\begin{proposition}[Weak maximal realizability]
For every $k\ge4$, there exists a bipartite polyhedral graph with bipartition classes of size $k$ and edge number $4k-4$.
\end{proposition}

\begin{proof}
Set $n=k-1$ and take the trapezohedron graph $T_n$. It has
\[
V=2n+2=2k
\]
vertices. Its edge set consists of the $2n$ equatorial cycle edges, $n$ edges from $N$ to the even equatorial vertices, and $n$ edges from $S$ to the odd equatorial vertices. Thus
\[
E=2n+n+n=4n=4k-4.
\]
It is bipartite with parts
\[
\{N,x_1,x_3,\ldots,x_{2n-1}\}
\]
and
\[
\{S,x_0,x_2,\ldots,x_{2n-2}\},
\]
each of size $n+1=k$. The displayed quadrilateral faces give a planar quadrangulation with all faces of length four, hence it attains the bipartite planar upper bound $E=2V-4$. As the graph of the convex $n$-gonal trapezohedron, it is polyhedral. Hence it is a weak balanced bipartite polyhedral realization attaining the upper Euler bound.
\end{proof}

For $k=4$, $n=3$ and $T_3$ is the cube graph. Thus the cube appears both as the unique strong realization and as the first member of the trapezohedral weak maximal family.

\begin{table}[h]
\centering
\begin{tabular}{ccccc}
\toprule
$k$ & $n=k-1$ & weak maximal graph & $V$ & $E=4k-4$ \\
\midrule
4 & 3 & triangular trapezohedron / cube & 8 & 12 \\
5 & 4 & 4-gonal trapezohedron & 10 & 16 \\
6 & 5 & 5-gonal trapezohedron & 12 & 20 \\
7 & 6 & 6-gonal trapezohedron & 14 & 24 \\
8 & 7 & 7-gonal trapezohedron & 16 & 28 \\
9 & 8 & 8-gonal trapezohedron & 18 & 32 \\
10 & 9 & 9-gonal trapezohedron & 20 & 36 \\
\bottomrule
\end{tabular}
\caption{Weak maximal bipartite polyhedral realizations supplied by trapezohedron graphs.}
\end{table}

\section{The Euler interval is not always realizable}

The Euler interval
\[
3k\le E\le4k-4
\]
is necessary, but it is not a complete existence criterion. The first intermediate case already fails.

\begin{theorem}[First empty interval case]
There is no bipartite polyhedral graph with $V=10$ vertices and $E=15$ edges. Equivalently, for $k=5$, the lower endpoint $E=3k=15$ of the Euler interval is not realizable. Dually, this is the nonexistence of an even triangulation of the sphere on seven vertices.
\end{theorem}

\begin{proof}
Suppose such a graph $G$ exists. Since $G$ is polyhedral, its minimum degree is at least three. Since $V=10$ and $E=15$, the degree sum is $2E=30=3V$. Hence every vertex has degree exactly three, so $G$ is cubic.

Euler's formula gives
\[
F=2-V+E=2-10+15=7.
\]
Since $G$ is a polyhedral graph, it is simple and 3-connected, and its planar dual $G^*$ is again simple and 3-connected. The dual graph $G^*$ is therefore a planar triangulation on seven vertices. Since $G$ is bipartite, all faces of $G$ have even length, so all vertices of $G^*$ have even degree. Thus $G^*$ is an even, or Eulerian, triangulation. The degree sum in $G^*$ is $2E=30$. In a triangulation, every vertex has degree at least three; here degrees are even, so every degree is at least four. The only way to write $30$ as the sum of seven even integers at least four is
\[
(6,4,4,4,4,4,4).
\]
Thus $G^*$ would have one vertex of degree six and six vertices of degree four.

Let $v$ be the degree-six vertex of $G^*$. Since $G^*$ is a simple 3-connected triangulation, the link of $v$ is a cycle; its six neighbors occur in cyclic order around $v$ and form a cycle $C_6$. Since $|V(G^*)|=7=1+6$, the vertex $v$ together with its six neighbors exhausts the vertex set of $G^*$. Hence the region outside the cycle $C_6$ contains no additional vertices and must be triangulated by diagonals of the hexagon alone. A triangulation of a hexagon has exactly three internal diagonals.

Each vertex of $C_6$ already has two incident edges on $C_6$ and one incident edge to $v$. Since its required degree in $G^*$ is four, each boundary vertex must be incident with exactly one internal diagonal of the hexagon. Hence the three internal diagonals would have to form a perfect matching on the six boundary vertices.

However, no set of three non-crossing diagonals forming a perfect matching on the six vertices of a hexagon can triangulate the hexagon. Indeed, by the two ears theorem for polygon triangulations \cite{Meisters1975}, every triangulation of a polygon with more than three vertices has at least two ears, and an ear vertex is incident with no internal diagonal. This contradicts the requirement that every boundary vertex of $C_6$ be incident with exactly one internal diagonal. Therefore no such triangulation $G^*$ exists, and hence no such graph $G$ exists.
\end{proof}

\begin{remark}
This theorem shows that Euler inequalities give a necessary capacity interval, but not a complete characterization of balanced bipartite polyhedral realizability. In dual language, the obstruction is the absence of an even triangulation on seven vertices. This places the example within the established theory of even triangulations and cubic bipartite planar graphs; see, for example, \cite{Batagelj1989,MatsumotoNakamoto2015,BrinkmannMcKay2007}. Thus the weak realization problem contains finer polyhedral obstructions beyond the strong cubic criticality theorem.
\end{remark}

\section{Preference-theoretic manifestation: mutually last-ranked pairs}\label{sec:preference-manifestation}

We now connect the cubic critical graph to the preference rankings. In the cube, every vertex has a unique non-neighbor on the opposite side. This non-adjacency relation is a perfect matching between the two bipartition classes.

\begin{definition}[Cube-distance compatibility]
A size-four complete bipartite preference structure is cube-distance-compatible if there is a strong cube realization such that, for every agent, the unique opposite-side non-neighbor is ranked last by that agent.
\end{definition}

\begin{definition}[Mutually last-ranked pair]
A pair $(m,w)\in M\times W$ is mutually last-ranked if $w$ is last in $m$'s preference order and $m$ is last in $w$'s preference order.
\end{definition}

\begin{proposition}[Characterization of cube-distance compatibility]
A size-four complete bipartite preference structure is cube-distance-compatible if and only if it has a perfect matching of mutually last-ranked pairs.
\end{proposition}

\begin{proof}
If a cube-distance-compatible realization is given, then the opposite-side non-neighbor relation in the cube is a perfect matching. By compatibility, each non-neighbor pair is mutually last-ranked.

Conversely, suppose the profile has a perfect matching $P$ of mutually last-ranked pairs. Since the non-neighbor relation between the two bipartition classes of $Q_3$ is a perfect matching, place the agents on the cube so that the pairs in $P$ are exactly the non-neighbor pairs. Then each agent ranks its unique non-neighbor last, so the profile is cube-distance-compatible.
\end{proof}

For a fixed mutually last-ranked perfect matching, each agent may order the remaining three opposite-side agents arbitrarily. Thus the number of labeled cube-distance-compatible size-four profiles is
\[
4!(3!)^8.
\]
This count is included only to illustrate the size of the compatible subfamily; a full isomorphism classification is left for later work.

\section{Discussion}

The main result identifies the cube as a critical mediator between complete bipartite preference capacity and polyhedral geometry. The strong condition attempts to realize, for each vertex, all but one opposite-side options by direct adjacency. Euler constraints permit this only at $k=4$, where the graph is $Q_3$.

Weak realizations behave differently. The trapezohedron family shows that maximal weak balanced bipartite polyhedral graphs exist for every $k\ge4$. However, the nonexistence of a bipartite polyhedral graph with $(V,E)=(10,15)$ shows that the Euler interval is not sufficient. This separates three levels: strong realization, weak maximal realization, and intermediate weak realizability.

There is also a natural duality picture. The cube is dual to the octahedron. More generally, the $n$-gonal trapezohedron is dual to the $n$-gonal antiprism. Thus the present construction may be viewed either in terms of bipartite quadrangulations or, dually, in terms of Eulerian triangulations.

The preference-theoretic section shows that the cube criticality has a direct rank manifestation: the non-adjacency matching of the cube corresponds precisely to mutually last-ranked pairs. Further refinements may study automorphism groups, side exchange, and stable-matching invariants of cube-compatible profiles, but those questions are not needed for the present polyhedral realizability theorem.

\section{Conclusion}

We formulated a polyhedral realizability problem for complete bipartite preference structures. The strong version has a unique critical solution: it exists if and only if $k=4$, and the underlying graph is the cube graph $Q_3$. The proof is an Euler-type capacity obstruction comparing the strong edge requirement $k(k-1)$ with the bipartite planar upper bound $4k-4$ and the polyhedral minimum-degree condition.

We also showed that weak maximal realizations exist for all $k\ge4$ via trapezohedron graphs, while the Euler interval is not fully realizable, as demonstrated by the nonexistence of a bipartite polyhedral graph with $V=10$ and $E=15$. Finally, in the critical cube case, we characterized distance-compatible preference profiles by the existence of a perfect matching of mutually last-ranked pairs.


\begin{thebibliography}{9}

\bibitem{GaleShapley1962}
D. Gale and L. S. Shapley, College admissions and the stability of marriage, \emph{American Mathematical Monthly} 69 (1962), 9--15.

\bibitem{GusfieldIrving1989}
D. Gusfield and R. W. Irving, \emph{The Stable Marriage Problem: Structure and Algorithms}, MIT Press, 1989.

\bibitem{Diestel2017}
R. Diestel, \emph{Graph Theory}, 5th ed., Graduate Texts in Mathematics, Springer, 2017.

\bibitem{BondyMurty2008}
J. A. Bondy and U. S. R. Murty, \emph{Graph Theory}, Graduate Texts in Mathematics, Springer, 2008.

\bibitem{Grunbaum2003}
B. Gr\"unbaum, \emph{Convex Polytopes}, 2nd ed., Graduate Texts in Mathematics, Springer, 2003.

\bibitem{Ziegler1995}
G. M. Ziegler, \emph{Lectures on Polytopes}, Graduate Texts in Mathematics, Springer, 1995.

\bibitem{Steinitz1922}
E. Steinitz, Polyeder und Raumeinteilungen, in \emph{Encyklop\"adie der mathematischen Wissenschaften}, 1922.

\bibitem{Harary1969}
F. Harary, \emph{Graph Theory}, Addison-Wesley, 1969.

\bibitem{Batagelj1989}
V. Batagelj, An improved inductive definition of two restricted classes of triangulations of the plane, \emph{Banach Center Publications} 25 (1989), 11--18.

\bibitem{MatsumotoNakamoto2015}
N. Matsumoto and A. Nakamoto, Generating 4-connected even triangulations on the sphere, \emph{Discrete Mathematics} 338 (2015), 64--70.

\bibitem{BrinkmannMcKay2007}
G. Brinkmann and B. D. McKay, Fast generation of planar graphs, \emph{MATCH Communications in Mathematical and in Computer Chemistry} 58 (2007), 323--357.

\bibitem{Meisters1975}
G. H. Meisters, Polygons have ears, \emph{American Mathematical Monthly} 82 (1975), 648--651.

\bibitem{Cromwell1997}
P. R. Cromwell, \emph{Polyhedra}, Cambridge University Press, 1997.
\end{thebibliography}
\end{document}